\documentclass[preprint,12pt]{elsarticle}

\usepackage{amssymb}
\usepackage{amsmath}
\usepackage{amsfonts}
\usepackage{graphicx}
\usepackage{graphics} 
\usepackage{tikz}
\newtheorem{theorem}{Theorem}

\newtheorem{corollary}{Corollary}
\newtheorem{proposition}{Proposition}
\newtheorem{definition}{Definition}

\newtheorem{example}{Example}
\newenvironment{proof}[1][Proof]{\noindent \textbf{#1.} }{\  \rule{0.5em}{0.5em}}

\journal{}

\begin{document}

\begin{frontmatter}
 
\title{Lexicographic products and lexicographic powers of graphs - a walk matrix approach}

\author{Domingos M. Cardoso$^{\dag}$ $^{\ddag}$, Paula Carvalho$^{\dag}$ $^{\ddag}$, Helena Gomes$^{\dag}$ $^{\S}$, Sofia J. Pinheiro$^{\dag}$ $^{\ddag}$, Paula Rama$^{\dag}$ $^{\ddag}$}

\address{$^{\dag}$CIDMA - Centro de Investiga\c{c}\~{a}o e Desenvolvimento em Matem\'{a}tica e Aplica\c{c}\~{o}es}

\address{$^{\ddag}$Departamento de Matem\'{a}tica, Universidade de Aveiro, Aveiro, Portugal.}
\address{$^{\S}$Escola Superior de Educa\c{c}\~{a}o, Instituto Polit\'ecnico de Viseu, Viseu, Portugal}

\begin{abstract}

\noindent  The characteristic polynomial and the spectrum of the lexicographic product of graphs $H[G]$, a specific instance of the generalized composition (also called $H$-join), are explicitly determined for arbitrary graphs $H$ and $G$, in terms of the eigenvalues of $G$ and an $H[G]$ associated matrix $\widetilde{{\bf W}}$, which relates $H$ with $G$. This study also establishes conditions under which a main eigenvalue of $G$ is a main or non-main eigenvalue of the matrix $\widetilde{{\bf W}}$, when the nullity of the graph $H$ is $\eta>0$. In such a case, we prove that every main eigenvalue of $G$ is an eigenvalue of $\widetilde{{\bf W}}$ with multiplicity at least $\eta$ which is non-main for $\bf \widetilde{W}$ if and only if $0$ is a non-main eigenvalue of $H$.
Furthermore, the spectra of the lexicographic powers of arbitrary graphs $G$ are analysed by applying the obtained results.
\medskip

\smallskip


\end{abstract}
\end{frontmatter}

\section{Introduction}

The generalized composition of a family of graphs $\mathcal{G}=\{G_1, \dots, G_p\}$, $H[G_1, \dots, G_p]$, where $H$ is a graph of order $p$,  was introduced in 1974 by Allen Schwenk \cite{Schwenk1974}. In \cite[Th. 7]{Schwenk1974}, assuming that $G_1, \dots, G_p$ are all regular graphs and taking into account that $V(G_1) \cup \dots \cup V(G_p)$ is an equitable partition $\pi$, the spectrum
of $H[G_1, \dots, G_p]$ is determined in terms of the spectra of the graphs $G_1, \dots, G_p$ and the matrix associated with $\pi$.

The generalized composition $H[G_1, \dots, G_p]$ was rediscovered in \cite{Cardoso_et_al2013} under the designation of $H$-join of a family of graphs $\mathcal{G}=\{G_1, \dots, G_p\}$ (see also \cite[Def. 2.5]{CardosoGomesPinheiro2022}).
Using a generalization of Fiedler's result \cite[Lem. 2.2]{Fiedler1974} obtained in \cite[Th. 3]{Cardoso_et_al2013}, the spectrum of the $H$-join of a family of regular graphs and the Laplacian spectrum of a family of arbitrary graphs are determined in \cite[Th. 5 and Th. 8]{Cardoso_et_al2013}.

In \cite{SaravananMuruganArunkumar2020}, as an application of a new generalization of Fiedler's lemma, the characteristic polynomial of the universal adjacency matrix of the $H$-join of a family of arbitrary graphs is determined in terms of the characteristic polynomials and rational functions of the components, and the determinant of an associated matrix whose diagonal entries are also rational functions. These rational functions are related to the main eigenvalues and main angles of the components of the $H$-join.\\

In \cite[p. 167]{Schwenk1974}, Schwenk wrote the following remark: In general, it does not appear likely that the characteristic polynomial of the generalized composition $H[G_1, \dots, G_p]$ can always be expressed in terms of the characteristic polynomials of the graphs $H, G_1, \dots, G_p$.\\

Recently, the problem raised by this remark was solved in \cite{CardosoGomesPinheiro2022}, considering an arbitrary graph $H$ with $p$ vertices and a family of arbitrary graphs $\mathcal{G}=\{G_1, \dots, G_p\}$, based on a walk matrix approach, the authors obtained an expression for the determination of the characteristic polynomial, as well as the entire spectrum of the universal adjacency matrix of the $H$-join in terms of the characteristic polynomials or spectra of the universal adjacency matrix of the graphs $G_1, \dots, G_p$ and an associated matrix which relates these graphs and the graph $H$.\\

The lexicographic product of the graphs $H$ and $G$, introduced by Harary in \cite{harary1} and Sabidussi in \cite{sabidussi} (see also \cite{hammack_et_al,harary2}), is denoted by $H[G]$ and is also called composition, has as vertex set of the cartesian product $V(H) \times V(G)$ and a vertex $(h_1,g_1)$ is adjacent to a vertex $(h_2,g_2)$ if either $h_1$ is adjacent to $h_2$ in $H$ or $h_1=h_2$ and $g_1$ is adjacent to $g_2$ in $G$ (see \cite{hammack_et_al,harary2}). It is immediate to conclude that the lexicographic product $H[G]$ coincides with the $H$-join of $\mathcal{G}$, where the graphs of the family $\mathcal{G}$ are all isomorphic to a fixed graph $G$. From the definition, it follows that this graph operation is not commutative. As referenced in \cite{barik_et_al}, if the graphs $F$ and $H$ are two nontrivial graphs with at least two vertices each, then $F[H]$ is connected if and only if $F$ is connected. Therefore, $F[H] \ncong H[F]$ (that is, $F[H]$ is not isomorphic to $H[F]$), whenever one of $F$ or $H$ is disconnected. Furthermore, there are connected graphs $F$ and $H$ such that $F[H] \ncong H[F]$. Considering the example given in \cite{barik_et_al}, $P_3[K_2] \ncong K_2[P_3]$, since $P_3[K_2]$ has $11$ edges and $K_2[P_3]$ has $13$ edges.

However, the lexicographic product is associative. In particular, is power associative, since the $k$-th lexicographic power of $H$, $H^k$, is the lexicographic product of $H$ by itself $k$ times (then $H^2=H[H], H^3=H[H^2]=H^2[H], \dots$).

In \cite{WangWong2018}, considering arbitrary graphs $H$ and $G$, the authors obtained the characteristic polynomial of the lexicographic product $H[G]$, not only related to $H$  and $G$, but also to eigenspaces of $G$. However, if $H$ has only one or two main eigenvalues, they show that the characteristic polynomial of $H[G]$ can be independent of eigenspaces of $G$. As an application, the spectrum of $H[G]$ is obtained explicitly when $G$ has exactly one or two main eigenvalues.\\

In this paper, we consider arbitrary graphs $H$ and $G$ of order $p$ and $n$, respectively. The characteristic polynomial and the spectrum of the lexicographic product $H[G]$ are explicitly determined in terms of the eigenvalues of $G$ and an associated matrix $\widetilde{{\bf W}}$ which relates $H$ with $G$. Additionally, we prove that every main eigenvalue of $G$ is an eigenvalue of $\widetilde{{\bf W}}$ with multiplicity at least the nullity of $H$. Furthermore, these eigenvalues are non-main for $\widetilde{{\bf W}}$ if and only if $0$ is a non-main eigenvalue of $H$. Finally, in the last section, the spectra of the lexicographic powers of arbitrary graphs $H^k$ are analysed by applying the obtained results.

\section{Preliminary results and concepts}

Throughout this paper, we will consider undirected simple and finite graphs $H$ and $G$, with vertex sets $V(H)=\{1, \dots, p\}$ and $V(G)=\{1,2,\ldots, n\}$ and edge sets $E(H)$ and $E(G)$, respectively. An edge linking the vertices $i$ and $j$ of $V(G)$ is denoted by $ij\in E(G)$, and in this case we say that $i$ and $j$ are \textit{adjacent}. For each vertex $i \in V(G),$ $N_{G}(i)$
denotes its \textit{neighbourhood}, that is the set of vertices of $G$ which are adjacent to~$i$ and $\left \vert N_{G}(i)\right \vert$ is called the \textit{degree} of  $i$ and denoted by $d_G(i)$.

The \textit{adjacency} matrix $A(G)=[a_{ij}]$ of $G$ is the symmetric matrix such that $a_{ij}=1$ if $ij\in E(G)$ and $0$, otherwise.  The multiset of eigenvalues of $A(G)$ (called the \textit{spectrum} of~$G$) is defined as $\sigma(G) = \big\{\mu_1^{[m_1]}, \mu_2^{[m_2]}, \ldots, \mu_t^{[m_t]}\big\}$, where $\mu_i^{[m_i]}$ means that the eigenvalue $\mu_i$ has algebraic multiplicity $m_i$. The \textit{eigenspace} of $\lambda \in \sigma(G)$, denoted $\mathcal{E}_G(\lambda)$, is defined as the kernel (or null space) of the matrix $A(G) - \lambda I_{n}$, where $I_{n}$ is the $n \times n$ identity matrix, and for any square matrix $B$, $\ker(B)$ represents its kernel.\\

We say that a graph $H$ has nullity $\eta$ when $0$ is an eigenvalue of $H$ with multiplicity $\eta$, that is, the eigenspace $\mathcal{E}_H(0)$ has dimension $\eta$.\\

Each of the distinct eigenvalues $\mu_{1}, \mu_{2}, \ldots, \mu_{s}$, with $s \le t$, of a graph $G$ whose eigenspace $\mathcal{E}_G(\mu _i)$ is not orthogonal to the all-1 vector with $n$ entries $\textbf{j}_n$ is said to be \textit{main}; otherwise, it is \textit{non-main}. The concept of main (non-main) eigenvalue was introduced by Cvetkovi\' c in \cite{cvetkov70} and was further investigated in several publications. Rowlinson's work \cite{rowmain} provides a survey on the main eigenvalues of a graph. The $H$-join \cite{Cardoso_et_al2013} of a family of finite undirected simple graphs $\mathcal{G}=\{G_1, \dots, G_p\}$, denoted $\bigvee_{H}{\mathcal{G}}$ (or, alternatively, $H[G_1, \dots, G_p]$), is the graph obtained from the graph $H$ replacing each vertex $i$ of $H$ by $G_i$ and adding to the edges of all graphs in $\mathcal{G}$ the edges of the join $G_i \vee G_j$, for every edge $ij \in E(H)$.
In \cite{CardosoGomesPinheiro2022, SaravananMuruganArunkumar2020} the spectrum of the universal adjacency matrix of the $H$-join of a family of arbitrary graphs $\mathcal{G}$ is determined.\\

In \cite[Ex. 3.4]{CardosoGomesPinheiro2022}, we considered the $H$-join $\bigvee_{P_3}{\{K_{1,3}, K_2, P_3\}}$, that is, $P_3[K_{1,3}, K_2, P_3]$. Taking into account that $\sigma(K_{1,3})=\{\sqrt{3},-\sqrt{3},0^{[2]}\},$ $\sigma(K_2)=\{1,-1\}$, $\sigma(P_3)=\{\sqrt{2},-\sqrt{2},0\}$ and the main characteristic polynomials of the graph components are $m_{K_{1,3}}(x) =  x^2 - 3$, $m_{K_2}(x) = x  - 1$ and $m_{P_3}(x) = x^2 - 2$, respectively, it was concluded that the characteristic polynomial of this $H$-join is
$$
\phi(P_3[K_{1,3}, K_2, P_3])=x^3(x+1)\phi({\bf \widetilde{W}})=x^3(x+1)(-42 - 40 x + 15 x^2 + 19 x^3 + x^4 - x^5),
$$
 where $\widetilde{\bf W}$ is an H-join associated matrix, as defined in \cite[Def. 3.1]{CardosoGomesPinheiro2022}, and recalled later in this text.\\

 The isomorphism relation between graphs is denoted $\cong$ and thus $H \cong G$ means that the graphs $H$ and $G$ are isomorphic.

\section{Lexicographic products of graphs}
Let us determine the spectrum of the lexicographic product $H[G]$, where $H$ is a graph of order $p$ and $G$ is a graph of order $n$ such that the main characteristic polynomial is $m_G(x)=x^s-\sum_{j=0}^{s-1}{c_jx^j}$. From now on, given a graph $H$, we consider the following notation
$$
\delta_{i,j}(H) = \left\{\begin{array}{ll}
                          1 & \hbox{if } ij \in E(H), \\
                          0 & \hbox{otherwise.}
                          \end{array}
                          \right.
$$

\begin{definition}\label{lex_product}
Consider the graph $H$ with vertex set $V(H)=\{1,2, \dots, p\}$ and the graph $G$ of order $n$. Then the lexicographic product $H[G]$ is the $H$-join of $\mathcal{G}$, where $\mathcal{G}$ is a family of $p$ graphs all of which are isomorphic to $G$.
\end{definition}

When the graph $G$ is regular, the spectrum of $H[G]$ is easily determined by direct application of \cite[Th. 5]{Cardoso_et_al2013}.

\begin{proposition}\label{prop_lex_product}
Let $H$ be a graph of order $p$ such that $\sigma(H)\!=\!\{\lambda_1^{[m_1]}, \dots, \lambda_t^{[m_t]}\}$, where $t \le p$ and $m_1 + \dots + m_t = p$, and let $G$ be a $k$-regular graph of order $n$. Then
$$
\sigma(H[G]) = \left(\bigcup_{j=1}^{p}{(\sigma(G_j)\setminus \{k\})}\right) \cup \sigma(\widetilde{C}),
$$
where $\widetilde{C} = \left(\begin{array}{cccc}
                                 k & \delta_{1,2}(H) n & \dots & \delta_{1,p}(H) n\\
                                   \delta_{2,1}(H) n & k & \dots & \delta_{2,p}(H) n\\
                                   \vdots    &  \vdots   &\ddots & \vdots\\
                                   \delta_{p,1}(H) n & \delta_{p,2}(H) n &\ddots & k\\
                            \end{array}\right) = n A(H) + k I_p$,
and thus $\sigma(\widetilde{C}) = \{(n\lambda_1+k)^{[m_1]}, \dots, (n\lambda_t+k)^{[m_t]}\}$.
\end{proposition}

\begin{example}
Let us assume that $H \cong K_{1,2}$ and $G \cong C_4$. Since $\sigma(H) = \{\sqrt{2}, 0, -\sqrt{2}\}$ and $\sigma(G)=\{2, 0^{[2]},-2\}$, applying Proposition~\ref{prop_lex_product}, we obtain
\begin{eqnarray*}
\sigma(H[G]) &=& \{0^{[6]}, (-2)^{[3]}\} \cup \sigma\left(4\left(\begin{array}{ccc}
                                                               0 & 1 & 0\\
                                                               1 & 0 & 1\\
                                                               0 & 1 & 0
                                                              \end{array}\right)+2\left(\begin{array}{ccc}
                                                                                         1 & 0 & 0\\
                                                                                         0 & 1 & 0\\
                                                                                         0 & 0 & 1                                                                             \end{array}\right)\right)\\
           &=& \{0^{[6]}, (-2)^{[3]}\} \cup 4\{\sqrt{2}, 0, -\sqrt{2}\} + 2\\
           &=& \{0^{[6]}, (-2)^{[3]}\} \cup \{4\sqrt{2}+2, 2, -4\sqrt{2}+2\}\\
           &=& \{2(1+2\sqrt{2}), 2, 0^{[6]}, 2(1-2\sqrt{2}), (-2)^{[3]}\}.
\end{eqnarray*}
\end{example}

For the application of the $H$-join to $H[G]$, we need the definition of the $H$-join associated matrix introduced in \cite{CardosoGomesPinheiro2022} in the particular case of $H[G]$ which is herein called $H[G]$ associated matrix.

\begin{definition}\label{main_def}
Consider the lexicographic product $H[G]$ as in Definition~\ref{lex_product}. Consider also that the main eigenvalues of $G$ are $\mu_{1}, \dots, \mu_{s}$ and the main characteristic polynomial is $m_{G}(x)=x^s - c_{0} - c_{1}x - \cdots - c_{s-1} x^{s-1}$. Let ${\bf W}_{G}$ be the walk matrix of $G$ and, denoting $\delta_{i,j}=\delta_{i,j}(H)$, consider the matrices
{\footnotesize $$
\hspace{-1cm}\widetilde{{\bf W}}_i = \left(\begin{array}{ccccccccccc}
\overbrace{\delta_{i,1}{\bf j}^T_{n}{\bf W}_{G}}^{s\text{ columns}}&\cdots&\overbrace{\delta_{i,i-1}{\bf j}^T_{n}{\bf W}_{G}}^{s\text{ columns}}&      0    &   0  &\cdots&  0   &  c_{0} & \overbrace{\delta_{i,{i+1}}{\bf j}^T_{n}{\bf W}_{G}}^{s\text{ columns}} &\cdots&\overbrace{\delta_{i,p}{\bf j}^T_{n}{\bf W}_{G}}^{s\text{ columns}}\\
                  {\bf 0}                                          &\cdots& {\bf 0}&      1    &   0  &\cdots&  0   &  c_{1}  &{\bf 0}&\cdots&      {\bf 0}    \\
                  {\bf 0}                                          &\cdots&{\bf 0}&      0    &   1  &\cdots&  0   &  c_{2}  &{\bf 0}&\cdots&      {\bf 0}    \\
                  \vdots                                           &\ddots&\vdots&   \vdots  &\vdots&\ddots&\vdots&  \vdots   &\vdots&\ddots&      \vdots     \\
                  {\bf 0}                                          &\cdots&{\bf 0}&      0    &   0  &\cdots&  1   &c_{s-1} &{\bf 0}&\cdots&      {\bf 0}
\end{array}\right),
$$}
for $1 \le i \le p$, where ${\bf 0}$ is a row vector with $s$ entries equal to $0$.  Then the $H[G]$ associated matrix is the $ps \times ps$ matrix
\begin{equation}\label{lex_prod_associated_matrix}
{\bf \widetilde{W}} = \left(\begin{array}{c}
                            {\bf \widetilde{W}}_1\\
                            {\bf \widetilde{W}}_2\\
                            \vdots          \\
                            {\bf \widetilde{W}}_p
                            \end{array}\right).
\end{equation}
\end{definition}

According to Definition~\ref{main_def}, the submatrices in \eqref{lex_prod_associated_matrix} are\\

{\small \begin{eqnarray*}
\widetilde{\bf W}_1 &=& \left(\begin{array}{ccccccccc}
0 &   0   &\dots &   0   & c_0   & \delta_{1,2}{\bf j}^T_n{\bf W}_G &\dots  &\delta_{1,p-1}{\bf j}^T_n{\bf W}_G & \delta_{1,p}{\bf j}^T_n{\bf W}_G\\
1 &   0   &\dots &   0   & c_1   &          {\bf 0}            &\dots  &       {\bf 0}      &      {\bf 0}    \\
0 &   1   &\dots &   0   & c_2   &          {\bf 0}            &\dots  &       {\bf 0}      &      {\bf 0}    \\
\vdots &\vdots &\ddots&\vdots &\vdots &      \vdots            &\ddots &        \vdots      &      \vdots     \\
0 &   0   &\dots &   1   &c_{s-1}&          {\bf 0}            &\dots  &        {\bf 0}     &      {\bf 0}
                        \end{array}\right),\\
\vdots         &\vdots& \hspace{6.2cm} \vdots \\
\widetilde{\bf W}_{p} &=& \left(\begin{array}{ccccccccc}
\delta_{p,1}{\bf j}^T_n{\bf W}_G &\delta_{p,2}{\bf j}^T_n{\bf W}_G &\dots & \delta_{p,p-1}{\bf j}^T_n{\bf W}_G & 0 & 0 &\dots& 0 & c_0 \\
          {\bf 0}                  &              {\bf 0}              &\dots & {\bf 0} & 1 & 0 &\dots& 0 & c_1 \\
          {\bf 0}                  &              {\bf 0}              &\dots & {\bf 0} & 0 & 1 &\dots& 0 & c_2 \\
          \vdots                   &               \vdots              &\ddots&\vdots&\vdots&\ddots&\vdots& \vdots&\vdots\\
          {\bf 0}                  &               {\bf 0}             &\dots & {\bf 0} & 0 & 0 &\dots & 1 &c_{s-1}
                        \end{array}\right),
\end{eqnarray*}}
and ${\bf C}(m_G) = \left(\begin{array}{ccccc}
                    0 &   0   &\dots &   0   & c_0 \\
                    1 &   0   &\dots &   0   & c_1 \\
                    0 &   1   &\dots &   0   & c_2 \\
                    \vdots &\vdots &\ddots&\vdots &\vdots\\
                    0 &   0   &\dots &   1   &c_{s-1} \\
                      \end{array}\right)$ is the Frobenius companion matrix of the main characteristic polynomial
$m_G(x) = x^s - c_0 - c_1 x - \dots - c_{s-1} x^{s-1}$, whose roots (that is, eigenvalues of ${\bf C}(m_{G})$) are the main eigenvalues of $G$.  Observe that the coefficients of the main characteristic polynomial $m_G(x)$ are easily determined using the walk matrix approach \cite{hagos1} (see also \cite[Cor. 2.3]{CardosoGomesPinheiro2022}).\\

While a square matrix $A$ has the same eigenvalues as $A^T$, its eigenvectors are, in general, different.
An eigenvector $\hat{\mathbf{u}}$ of $A^T$, satisfying $A^T\hat{\mathbf{u}} = \lambda \hat{\mathbf{u}}$
is usually referred to as the left eigenvector of $A$, since $\hat{\mathbf{u}}^TA = \lambda \hat{\mathbf{u}}^T$.

Returning to the matrix ${\bf C}(m_{G})$, since the roots of $m_{G}(x)$, $\mu_{1}, \dots, \mu_{s}$, are all distinct, ${\bf C}(m_{G})$ is diagonalizable, that is,
$$
{\bf U}{\bf C}(m_{G}) {\bf U}^{-1} = \text{diag}\left(\mu_{1}, \dots, \mu_{s}\right),
$$
where ${\bf U}$ is a matrix whose rows, $\hat{\bf u}^T_{\mu_{i}}$, with $1 \le i \le s$, are the left eigenvectors of ${\bf C}(m_{G})$.\\
Now, it is worth to recall that for every $\mu \in \{\mu_{1}, \dots, \mu_{s}\}$ the vector $\hat{\bf u}^T_{\mu} = (1, \mu, \mu^2, \dots, \mu^{s-1})$ is a left eigenvector of ${\bf C}(m_{G})$. Indeed, it follows that
$$
(1, \mu, \mu^2, \dots, \mu^{s-1})\underbrace{\left(\begin{array}{ccccc}
                                  0 &   0   &\dots &   0   & c_0 \\
                                  1 &   0   &\dots &   0   & c_1 \\
                                  0 &   1   &\dots &   0   & c_2 \\
                             \vdots &\vdots &\ddots&\vdots &\vdots\\
                                  0 &   0   &\dots &   1   &c_{s-1} \\
                                 \end{array}\right)}_{{\bf C}(m_{G})} = \mu(1, \mu, \dots, \mu^{s-2}, \mu^{s-1})
$$

taking into account that $\sum_{j=0}^{s-1}{c_j\mu^j}=\mu^s.$

Defining ${\bf M}=\left(\begin{array}{c}
               {\bf j}^T_n{\bf W}_G \\
                 0    \; \dots \; 0 \\
               \vdots \; \ddots\; \vdots\\
                 0    \; \dots \; 0 \\
               \end{array}\right)$ which is an $s \times s$ submatrix of every $s \times ps$ matrix $\widetilde{\bf W}_{i}$, it follows that
\begin{equation}\label{associated_matrix}
\widetilde{\bf W} = \left(\begin{array}{ccccc}
                       {\bf C}(m_G)   & \delta_{1,2}{\bf M} & \dots & \delta_{1,p-1}{\bf M} & \delta_{1,p}{\bf M} \\
                        \delta_{2,1}{\bf M} & {\bf C}(m_G)  & \dots & \delta_{2,p-1}{\bf M} & \delta_{2,p}{\bf M} \\
                        \vdots        &    \vdots     &\ddots & \vdots          & \vdots \\
                        \delta_{p,1}{\bf M} &  \delta_{p,2}{\bf M} & \dots & \delta_{p,p-1}{\bf M} & {\bf C}(m_G)
                  \end{array}\right).
\end{equation}

\begin{theorem}\label{h-join_spectra_theorem}
Let $H[G]$ be the lexicographic product as in Definition~\ref{lex_product} and consider that
$\mu_1, \dots, \mu_s, \mu_{s+1}, \dots, \mu_t$ with $t \le n$, are the distinct eigenvalues of $G$,
where $\mu_1, \dots, \mu_s$ are the main eigenvalues and each eigenvalue $\mu_j$ has multiplicity $m_j$,
for $1 \le j \le t$. Then
\begin{equation}\label{h-join_spectra}
\sigma(H[G]) = \{\mu_{1}^{[p(m_1-1)]}, \dots, \mu_{s}^{[p(m_s-1)]}\} \cup\\
               \{\mu_{s+1}^{[p m_{s+1}]}, \dots, \mu_{t}^{[p m_t]}\} \cup \sigma({\bf \widetilde{W}}),
\end{equation}
where the union of multisets is with possible repetitions.
\end{theorem}

\begin{proof}
Applying \cite[Th. 3.2]{CardosoGomesPinheiro2022} to the lexicographic product $H[G]$, the result follows.
\end{proof}

What follows, we prove that every main eigenvalue of $G$ is an eigenvalue of $\widetilde{\bf W}$ with multiplicity at least the nullity of $H$.

\begin{theorem}\label{left_eigenvecors_theorem}
Let $\mu$ be a main eigenvalue of $G$ and $\hat{\bf u}_{\mu}$ the corresponding main eigenvector (see \cite{CardosoGomesPinheiro2022}). Then $\mu$ is an eigenvalue of ${\bf \widetilde{W}}$ with an associated left eigenvector
$\hat{\bf v}^T_{\mu} = \left(\gamma_1 \hat{\bf u}^T_{\mu}, \dots, \gamma_p \hat{\bf u}^T_{\mu}\right)$ if and
only if $\hat{\bf \gamma}^T = (\gamma_1, \dots, \gamma_p)$ is an eigenvector of $A(H)$ associated with $0 \in \sigma(H)$.
Furthermore, the multiplicity of $\mu$ as an eigenvalue of the matrix $\widetilde{\bf W}$ is at least the nullity of $H$.
\end{theorem}

\begin{proof}
Let $\mu$ be a root of $m_G(x)$ and thus an eigenvalue of ${\bf C}(m_G)$. Then the row vector
$\hat{\bf u}^T_{\mu} = (1, \mu, \mu^2, \dots, \mu^{s-1})$  is a left eigenvector of ${\bf C}(m_G)$, that is,
$\hat{\bf u}^T_{\mu} {\bf C}(m_G)  = \mu \hat{\bf u}^T_{\mu}$. Defining
\begin{equation}
\hat{\bf v}^T_{\mu} = \left(\gamma_1 \hat{\bf u}^T_{\mu}, \dots, \gamma_p \hat{\bf u}^T_{\mu}\right),\label{left_eigenvector_1}
\end{equation}
where $\gamma_1, \dots, \gamma_p$ are scalars, it follows that
\begin{eqnarray}
\hat{\bf v}^T_{\mu}\widetilde{\bf W}
&=& \left(\gamma_1\hat{\bf u}^T_{\mu}{\bf C}(m_G)+\sum_{i \in N_H(1)}{\gamma_i\hat{\bf u}^T_{\mu}{\bf M}}, \dots, \gamma_p\hat{\bf u}^T_{\mu}{\bf C}(m_G)+\sum_{i \in N_H(p)}{\gamma_i\hat{\bf u}^T_{\mu}{\bf M}}\right) \nonumber\\
                            &=& \mu\left(\gamma_1{\bf u}^T_{\mu}, \dots,\gamma_p{\bf u}^T_{\mu}\right) +
                                   \left(\sum_{i \in N_H(1)}{\gamma_i}{\bf j}^T_n{\bf W}_G , \dots ,
                                               \sum_{i \in N_H(p)}{\gamma_i}{\bf j}^T_n{\bf W}_G\right) \nonumber\\
                            &=& \mu \hat{\bf v}^T_{\mu} + {\bf j}^T_n{\bf W}_G\underbrace{\left(\sum_{i \in N_H(1)}{\gamma_i}, \dots,\sum_{i \in N_H(p)}{\gamma_i}\right)}_{(**)}. \label{left_eigenvalue_eq_1}
\end{eqnarray}
Therefore, since ${\bf j}^T_n{\bf W}_G$ is a row vector with positive components, it follows that
$\mu$ is an eigenvalue of $\widetilde{\bf W}$ with an associated left eigenvector $\hat{\bf v}_{\mu}$, defined in \eqref{left_eigenvector_1}, if and only if (**) is equal to zero, that is, if and only if
$\hat{\bf \gamma}^T = (\gamma_1, \gamma_2, \dots, \gamma_p)$ is a nontrivial solution of the linear system
\begin{equation}\label{null_eigebvector}
\left\{\begin{array}{ccc}
       \sum_{i \in N_H(1)}{x_i} &  =   & 0 \\
       \sum_{i \in N_H(2)}{x_i} &  =   & 0 \\
       \vdots                        &\vdots& \vdots \\
       \sum_{i \in N_H(p)}{x_i} &  =   & 0
\end{array}\right. \Leftrightarrow A(H) \hat{\bf x} = \hat{\bf 0}_p,
\end{equation}
where $\hat{\bf x}^T = (x_1, x_2, \dots, x_p)$ is the vector of unknowns.
Equivalently, we may say that $\mu$ is an eigenvalue of $\widetilde{\bf W}$ with an associated left eigenvector $\hat{\bf v}_{\mu}$ if and only if $\hat{\bf \gamma}$ is an eigenvector of $A(H)$ associated with $0 \in \sigma(H)$.\\

Let us assume that $H$ has nullity $\eta \ge 1$, then there are $\eta$ linearly independent solutions of the system \eqref{null_eigebvector} which are linearly independent eigenvectors $\hat{\gamma}_1, \dots, \hat{\gamma}_{\eta}$ forming a basis for $\mathcal{E}_H(0)$ and producing $\eta$ distinct eigenvectors $\hat{\bf v}_{\mu,1}, \dots, \hat{\bf v}_{\mu,\eta}$ of $\widetilde{\bf W}$ associated  with $\mu$. Furthermore, considering
$\hat{\bf \gamma}_j^T = (\gamma_{1,j}, \dots, \gamma_{p,j})^T$, for $1 \le j \le \eta$, the vectors
$\hat{\bf v}_{\mu,1}^T=\left(\gamma_{1,1} \hat{\bf u}^T_{\mu}, \dots, \gamma_{p,1} \hat{\bf u}^T_{\mu}\right)$,
$\dots$,  $\hat{\bf v}_{\mu,\eta}^T=\left(\gamma_{1,\eta} \hat{\bf u}^T_{\mu}, \dots, \gamma_{p,\eta} \hat{\bf u}^T_{\mu}\right)$ are linearly independent. Indeed,  if there are scalars $\beta_1, \dots, \beta_{\eta}$ not all zero such that $\sum_{j=1}^{\eta}{\beta_j\hat{\bf v}_{\mu,j}} = \hat{\bf 0}$, then
\begin{eqnarray*}
\sum_{j=1}^{k}{\beta_j\hat{\bf v}_{\mu,j}} &=& \sum_{j=1}^{\eta}{\beta_j\left(\begin{array}{c}
                                                                     \gamma_{1,j}\hat{\bf u}_{\mu} \\
                                                                     \vdots \\
                                                                     \gamma_{p,j}\hat{\bf u}_{\mu}
                                                                     \end{array}\right)} \; = \;
                                        \left(\begin{array}{c}
                                        \left(\sum_{j=1}^{\eta}{\beta_j\gamma}_{1,j}\right)\hat{\bf u}_{\mu}\\
                                        \vdots \\
                                        \left(\sum_{j=1}^{\eta}{\beta_j\gamma}_{p,j}\right)\hat{\bf u}_{\mu}
                                        \end{array}\right) \: = \: \left(\begin{array}{c}
                                                                         \hat{\bf 0}_s\\
                                                                         \vdots \\
                                                                         \hat{\bf 0}_s
                                                                         \end{array}\right)
\end{eqnarray*}
which is equivalent to write
\begin{eqnarray}
\left(\begin{array}{c}
      \sum_{j=1}^{\eta}{\beta_j\gamma}_{1,j}\\
      \vdots \\
      \sum_{j=1}^{\eta}{\beta_j\gamma}_{p,j} \\
      \end{array}\right) = \left(\begin{array}{c}
                                  0 \\
                                 \vdots \\
                                  0
                                 \end{array}\right) & \Leftrightarrow &
                      \sum_{j=1}^{\eta}{\beta_j\left(\begin{array}{c}
                                                  \gamma_{1,j} \\
                                                  \vdots \\
                                                  \gamma_{p,j}
                                                  \end{array}\right)} = \left(\begin{array}{c}
                                                                               0 \\
                                                                               \vdots \\
                                                                                0
                                                                               \end{array}\right) \nonumber \\
                                                   & \Leftrightarrow &
                      \sum_{j=1}^{\eta}{\beta_j\hat{\gamma}_j} = {\bf 0}_p.\label{lin_indp}
\end{eqnarray}
From \eqref{lin_indp} it follows that $\beta_1 = \dots = \beta_{\eta} = 0$, since $\hat{\gamma}_1, \dots, \hat{\gamma}_{\eta}$
are linearly independent.
\end{proof}

\begin{corollary}\label{corollary_1}
Let us assume that $H$ has nullity $\eta$ and let $\mu$ be a main eigenvalue of $G$. Then $\mu$ is an eigenvalue of the $H[G]$ associated matrix $\widetilde{\bf W}$ with multiplicity at least $\eta$, which is non-main if and only if $0$ is a non-main eigenvalue of $H$.
\end{corollary}

\begin{proof}
The first part is an immediate consequence of Theorem~\ref{left_eigenvecors_theorem}. Furthermore, taking into account that the $\eta$ linearly independent left eigenvectors of $\widetilde{\bf W}$ associated with $\mu$ are the vectors ${\bf \hat{v}}^{T}_{\mu}$ represented in \eqref{left_eigenvector_1}, where the scalars $\gamma_1, \dots, \gamma_p$ are the components of the linearly independent eigenvectors $\hat{\bf \gamma}^T = (\gamma_1, \gamma_2, \dots, \gamma_p)$ of the adjacency matrix of $H$ associated with the eigenvalue $0$, it follows that
\begin{eqnarray}
{\bf \hat{v}}^T_{\mu}\mathbf{j} &=& \sum_{i=1}^{p}{\gamma_i \hat{\bf u}^T_{\mu}\mathbf{j}_n} \nonumber\\
                          &=& \hat{\bf u}^T_{\mu}\mathbf{j}_n\left(\sum_{i=1}^{p}{\gamma_i}\right), \label{marca_1}
\end{eqnarray}
where $\mathbf{j}$ and $\mathbf{j}_n$ are the all-1 vectors of $\mathbb{R}^{pn}$ and $\mathbb{R}^n$, respectively.
Therefore, since $\hat{\bf u}^T_{\mu}\mathbf{j}_n \ne 0$, the expression in \eqref{marca_1} is equal to zero if and only if $\hat{\bf \gamma}^T = (\gamma_1, \dots, \gamma_p)$ is orthogonal to the all-1 vector, that is, $0$ is a non-main eigenvalue of $H$.
\end{proof}

\begin{example}\label{ex_2}
Assuming that $H \cong K_{2,2}$ and $G \cong K_{1,2}$, it follows that
$A(G)=\left(\begin{array}{ccc}
                   0 & 1 & 0\\
                   1 & 0 & 1\\
                   0 & 1 & 0
            \end{array}\right)$ and then we may determine the main polynomial $m_G(\lambda)$ using the walk matrix approach which takes into account that
            \begin{eqnarray*}
            m_G(A(G)) = \hat{0} & \Leftrightarrow & \sum_{j=0}^{s-1}{c_jA(G)^j{\bf j}} = A(G)^s{\bf j}\\
                                & \Leftrightarrow & \underbrace{(A(G)^0{\bf j}_3, A(G){\bf j}_3, \dots, A(G)^{s-1}{\bf j}_3)}_{W_G}\hat{c} = A(G)^s{\bf j}_3,
            \end{eqnarray*}
            where $\hat{c}^T=(c_0, c_1, \dots, c_{s-1})$. So, in the particular case of this example,
\begin{eqnarray*}
\underbrace{({\bf j}_3, A(G){\bf j}_3)}_{{\bf W_G}} \hat{c} = A(G)^2{\bf j}_3 \Leftrightarrow
                 \left(\begin{array}{cc}
                  1  &  1\\
                  1  &  2\\
                  1  &  1
                 \end{array}\right)\left(\begin{array}{c}
                                          c_0\\
                                          c_1
                                         \end{array}\right) &=& \left(\begin{array}{c}
                                                                       2\\
                                                                       2\\
                                                                       2\\
                                                                      \end{array}\right).
\end{eqnarray*}
Therefore, $m_G(\mu) = \mu^2 - c_1\mu - c_0 = \mu^2 - 2$ and then $\sigma(G) = \{\underbrace{\sqrt{2}, -\sqrt{2}}_{main}, \!\!\underbrace{0}_{non-main}\!\!\}$. Since the matrices ${\bf C}(m_G)$ and ${\bf M}$ are
$$
{\bf C}(m_G) = \left(\begin{array}{cc}
                            0 & c_0\\
                            1 & c_1
                     \end{array}\right) =  \left(\begin{array}{cc}
                                                  0 & 2\\
                                                  1 & 0
                                           \end{array}\right)  \text{ and }
                                                               {\bf M} = \left(\begin{array}{c}
                                                                              {\bf j}_3 {\bf W}_G\\
                                                                              0 \; \; \; 0
                                                                              \end{array}\right) = \left(\begin{array}{cc}
                                                                                                          3  &  4 \\
                                                                                                          0  &  0
                                                                                                         \end{array}\right),
$$
taking into account the adjacency matrix $A(H)$, we obtain $\delta_{1,2}(H) = \delta_{2,1}(H) = \delta_{3,4}(H) = \delta_{4,3}(H) = 0$ and $\delta_{i,j}(H)=1$,  $$ \text{for every } (i,j) \in \{(1,3),(3,1), (1,4),(4,1), (2,3),(3,2),(2,4),(4,2)\}.$$ Then the $H[G]$ associated matrix is
\begin{eqnarray*}
\widetilde{\bf W} &=& \left(\begin{array}{cccc}
                       {\bf C}(m_G)               & \hat{{\bf 0}}_{2 \times 2} & {\bf M}      & {\bf M}\\
                       \hat{{\bf 0}}_{2 \times 2} & {\bf C}(m_G)               & {\bf M}      & {\bf M}\\
                       {\bf M}                    & {\bf M}                    & {\bf C}(m_G) & \hat{{\bf 0}}_{2 \times 2}\\
                       {\bf M}                    & {\bf M}                    & \hat{{\bf 0}}_{2 \times 2} & {\bf C}(m_G)
                      \end{array}\right)\\
                  &=&  \left(\begin{array}{cccccccc}
                           0 & 2 & 0 & 0 & 3 & 4 & 3 & 4\\
                           1 & 0 & 0 & 0 & 0 & 0 & 0 & 0\\
                           0 & 0 & 0 & 2 & 3 & 4 & 3 & 4\\
                           0 & 0 & 1 & 0 & 0 & 0 & 0 & 0\\
                           3 & 4 & 3 & 4 & 0 & 2 & 0 & 0\\
                           0 & 0 & 0 & 0 & 1 & 0 & 0 & 0\\
                           3 & 4 & 3 & 4 & 0 & 0 & 0 & 2\\
                           0 & 0 & 0 & 0 & 0 & 0 & 1 & 0
                          \end{array}\right).
\end{eqnarray*}
Taking into account that $\sigma(H) = \{\underbrace{2}_{main}, \underbrace{-2,0^{[2]}}_{non-main}\}$, it follows that $H$ has nullity $\eta = 2$ and the eigenvalue $0$ is non-main. Applying Corollary~\ref{corollary_1}, we may conclude that the main eigenvalues of $G$, $\sqrt{2}$ and $-\sqrt{2}$, are non-main eigenvalues of $\widetilde{\bf W}$ each one with multiplicity $2$. So, from Theorem~\ref{h-join_spectra_theorem}, it follows
\begin{eqnarray*}
\sigma(H[G]) &=& \{\sqrt{2}^{[4 \times 0]}, (-\sqrt{2})^{[4 \times 0]}\} \cup \{0^{[4 \times 1]}\} \cup \sigma(\widetilde{\bf W})\\
             &=& \{0^{[4]}\} \cup \{\sqrt{2}^{[2]}, (-\sqrt{2})^{[2]}\} \cup \Lambda.
\end{eqnarray*}
Since the characteristic polynomial of the matrix $\widetilde{\bf W}$ is $\phi_{\widetilde{\bf W}}(x) = (-2 + x^2)^2 (-10 - 6 x + x^2) (6 + 6 x + x^2)$ it follows that $\Lambda = \{3+\sqrt{19}, 3-\sqrt{19}, 3 + \sqrt{3}, 3 - \sqrt{3}\}$.
\end{example}

\section{Lexicographic powers of a graph}

Consider a graph $G$ of order $n$, with distinct eigenvalues $\mu_{1(1)}, \dots, \mu_{s(1)},$ $\mu_{s(1)+1}, \dots, \mu_{t(1)}$ with $t(1) \le n$, where $\mu_{1(1)}, \dots, \mu_{s(1)}$ are the main eigenvalues and each $\mu_{j(1)}$ has multiplicity $m_{j(1)}$, for $1 \le j(1) \le t(1)$. Let us denote the distinct eigenvalues of the graph $G^k$ by $\mu_{1(k)}, \dots, \mu_{s(k)}, \mu_{s(k)+1}, \dots, \mu_{t(k)}$ with $t(k) \le n^k$, where $\mu_{1(k)}, \dots, \mu_{s(k)}$ are the main eigenvalues and each $\mu_{j(k)}$ has multiplicity $m_{j(k)}$.

Applying Theorem~\ref{h-join_spectra_theorem}, it follows that the spectrum of $G^{k+1}=G^k[G]=G[G^k]$, is
\begin{equation*}
\sigma(G[G^k])\!\! =\!\! \{\mu_{1(k)}^{[n(m_{1(k)}-1)]}, \dots, \mu_{s(k)}^{[n(m_{s(k)}-1)]}\} \cup \{\mu_{s(k)+1}^{[nm_{s(k)+1}]}, \dots, \mu_{t(k)}^{[nm_{t(k)}]}\} \cup \sigma(\widetilde{\bf W}(G^k)),
\end{equation*}
where
 $$
\widetilde{\bf W}(G^k)=\left(\begin{array}{ccccc}
                       {\bf C}(m_{G^k}) & \delta_{1,2}{\bf M}(G^k)& \dots & \delta_{1,n-1}{\bf M}(G^k) & \delta_{1,n}{\bf M}(G^k)\\
                       \delta_{2,1}{\bf M}(G^k)& {\bf C}(m_{G^k}) & \dots & \delta_{2,n-1}{\bf M}(G^k) & \delta_{2,n}{\bf M}(G^k)\\
                       \vdots                   &    \vdots       &\ddots & \vdots                     & \vdots \\
                       \delta_{n,1}{\bf M}(G^k) & \delta_{n,2}{\bf M}(G^k)& \dots & \delta_{n,n-1}{\bf M}(G^k) & {\bf C}(m_{G^k})
                       \end{array}\right),
$$
is an $ns(k) \times ns(k)$ matrix, ${\bf C}(m_{G^k}) = \left(\begin{array}{ccccc}
                                                   0 &   0   &\dots &   0   & c_{0(k)} \\
                                                   1 &   0   &\dots &   0   & c_{1(k)} \\
                                                   0 &   1   &\dots &   0   & c_{2(k)} \\
                                                   \vdots &\vdots &\ddots&\vdots &\vdots\\
                                                   0 &   0   &\dots &   1   &c_{s(k)-1} \\
                                                   \end{array}\right)$ is the Frobenius companion matrix of the main characteristic polynomial $m_{G^k}(x) = x^{s(k)} - c_{0(k)} - c_{1(k)} x - \dots - c_{s(k)-1} x^{s(k)-1}$,
${\bf M}(G^k)=\left(\begin{array}{c}
               {\bf j}^T_{n^k}{\bf W}_{G^k} \\
               0    \; \dots \; 0 \\
               \vdots \; \ddots\; \vdots\\
               0    \; \dots \; 0 \\
               \end{array}\right)$ and ${\bf W}_{G^k}$ is the walk matrix of $G^k$, that is, ${\bf W}_{G^k} = ({\bf j}_{n^k}, A(G^k){\bf j}_{n^k}, \dots, A(G^k)^{s_k-1}{\bf j}_{n^k})$. Observe that the first row of the matrix ${\bf M}(G^k)$ is ${\bf j}^T_{n^k}{\bf W}_{G^k}$ and the remaining entries are equal to zero.

Applying Corollary~\ref{corollary_1}, we may conclude that if the nullity of $G$ is $\eta >0$, then every main eigenvalue $\mu$ of $G$ is an eigenvalue of $\widetilde{\bf W}(G^k)$  with multiplicity at least $\eta$. Furthermore, these eigenvalues are non-main for $\widetilde{{\bf W}}$ if and only if $0$ is non-main as eigenvalue of $G$.

\begin{example}
Consider the graph $G \cong K_{1,2}$ and then $A(G) = \left(\begin{array}{ccccc}
                                                                   0 &   1   &   1 \\
                                                                   1 &   0   &   0 \\
                                                                   1 &   0   &   0
                                                            \end{array}\right).$
Using a walk matrix approach to determine the coefficients of the main polynomial $m_G(x)$ (see \cite{hagos1}), from
${\bf W}_G=({\bf j}, A(G){\bf j})=\left(\begin{array}{cc}
                                   1 & 2 \\
                                   1 & 1 \\
                                   1 & 1 \\
                            \end{array}\right)$ and taking into account that                             
                            $$m_G(x)=x^2-c_{0(1)}-c_{1(1)}x, m_G(A(G)){\bf j}_3 = \hat{\bf 0}_3 \Leftrightarrow A^2(G){\bf j}_{3} = c_{0(1)}{\bf j}+c_{1(1)}A(G){\bf j}_3,$$ it follows 
$$
A^2(G){\bf j}_3 = {\bf W}_G\left(\begin{array}{c}
                                c_0\\
                                c_1
                         \end{array}\right) \Leftrightarrow \left(\begin{array}{c}
                                                                         2 \\
                                                                         2 \\
                                                                         2
                                                                   \end{array}\right) = \left(\begin{array}{cc}
                                                                                                     1 & 2 \\
                                                                                                     1 & 1 \\
                                                                                                     1 & 1 \\
                                                                                              \end{array}\right)\left(\begin{array}{c}
                                                                                                                             c_{0(1)}\\
                                                                                                                             c_{1(1)}
                                                                                                                      \end{array}\right).
$$
Therefore, $c_{0(1)}=2, \; c_{1(1)}=0$ and thus $m_G(x) = x^2 - 2$. Furthermore, $\sigma(G)=\{\underbrace{\sqrt{2},-\sqrt{2}}_{main}, \underbrace{0}_{non-main}\}$.
\begin{enumerate}
\item \textbf{Let us determine} $\sigma(G^2)$.
      Taking into account that ${\bf M}(G) = \left(\begin{array}{cc}
                                                 3 & 4 \\
                                                 0 & 0
                                          \end{array}\right)$ and $C(m_G) = \left(\begin{array}{cc}
                                                                                         0 & 2 \\
                                                                                         1 & 0
                                                                                  \end{array}\right),$ it follows that
      $\widetilde{\bf W}(G) = \left(\begin{array}{cccccc}
                                        0 & 2 & 3 & 4 & 3 & 4 \\
                                        1 & 0 & 0 & 0 & 0 & 0 \\
                                        3 & 4 & 0 & 2 & 0 & 0 \\
                                        0 & 0 & 1 & 0 & 0 & 0 \\
                                        3 & 4 & 0 & 0 & 0 & 2 \\
                                        0 & 0 & 0 & 0 & 1 & 0
                                 \end{array}\right)$.\\
Taking into account that $-\sqrt{2}$ and $\sqrt{2}$ are main eigenvalues of $G$ and $0$ is a non-main eigenvalue with multiplicity $1$, by Corollary~\ref{corollary_1}, $-\sqrt{2}$ and $\sqrt{2}$ are non-main eigenvalues of $\widetilde{\bf W}$ with multiplicity $1$. Then
\begin{eqnarray*}
\phi(\widetilde{\bf W}(G)) &=& 56 + 96x + 16x^2 - 48x^3 -24x^4 + x^6\\
                           &=& \underbrace{(-2+x^2)}_{non-main}\underbrace{(-28 -48x - 22x^2 + x^4)}_{m_{G^2}(x)}.\\
\end{eqnarray*}
Denoting the characteristic polynomial of a graph $G$ (or a matrix $A$) by $\phi(G)$ ($\phi(A)$, respectively) and the multiset of roots of a polynomial $\phi$ by $\text{Roots}(\phi)$, applying Theorem~\ref{h-join_spectra_theorem}, we obtain
\begin{eqnarray*}
\sigma(G^2) &=& \{-\sqrt{2}^{[0]}, \sqrt{2}^{[0]}, 0^{[3]}\} \cup \sigma(\widetilde{\bf W}(G))\\
            &=& \{-\sqrt{2}^{[1]}, \sqrt{2}^{[1]}, 0^{[3]}\} \cup \{\text{Roots}(m_{G^2}(x))\},\\
\end{eqnarray*}
where $\text{Roots}(m_{G^2}(x))$ is
\begin{scriptsize}
$$\left\{-\frac{3}{\sqrt{2}} - \sqrt{\frac{1}{2}(13-8\sqrt{2})}, -\frac{3}{\sqrt{2}} + \sqrt{\frac{1}{2}(13-8\sqrt{2})}, \frac{3}{\sqrt{2}} - \sqrt{\frac{1}{2}(13+8\sqrt{2})}, \frac{3}{\sqrt{2}} + \sqrt{\frac{1}{2}(13+8\sqrt{2})}\right\}.$$
\end{scriptsize}
Furthermore
\begin{equation*}
\phi(G^2) = \underbrace{x^3(-2+x^2)}_{non-main}\underbrace{(-28 -48x - 22x^2 + x^4)}_{main}.
\end{equation*}

\item \textbf{{Let us determine} $\sigma(G^3)=\sigma(G[G^2]$)}.
Using again a walk matrix approach to determine the coefficients of the main polynomial $m_{G^2}(x)$ \cite{hagos1}, taking into account that $A(G^2) = \left(\begin{array}{ccccccccc}
                                    0 & 1 & 1 & 1 & 1 & 1 & 1 & 1 & 1 \\
                                    1 & 0 & 0 & 1 & 1 & 1 & 1 & 1 & 1 \\
                                    1 & 0 & 0 & 1 & 1 & 1 & 1 & 1 & 1 \\
                                    1 & 1 & 1 & 0 & 1 & 1 & 0 & 0 & 0 \\
                                    1 & 1 & 1 & 1 & 0 & 0 & 0 & 0 & 0 \\
                                    1 & 1 & 1 & 1 & 0 & 0 & 0 & 0 & 0 \\
                                    1 & 1 & 1 & 0 & 0 & 0 & 0 & 1 & 1 \\
                                    1 & 1 & 1 & 0 & 0 & 0 & 1 & 0 & 0 \\
                                    1 & 1 & 1 & 0 & 0 & 0 & 1 & 0 & 0 \\
                              \end{array}\right)$, we obtain
$
{\bf W}_{G^2}=({\bf j}, A(G^2){\bf j}, A^2(G^2){\bf j}, A^3(G^2){\bf j}) = \left(\begin{array}{cccc}
                                                                          1 & 8 & 40 & 236 \\
                                                                          1 & 7 & 34 & 208 \\
                                                                          1 & 7 & 34 & 208 \\
                                                                          1 & 5 & 30 & 162 \\
                                                                          1 & 4 & 27 & 138 \\
                                                                          1 & 4 & 27 & 138 \\
                                                                          1 & 5 & 30 & 162 \\
                                                                          1 & 4 & 27 & 138 \\
                                                                          1 & 4 & 27 & 138
                                                                          \end{array}\right)
$
and the main polynomial of $G^2$ is $m_{G^2}(x) = x^4 - c_{0(2)} - c_{1(2)}x - c_{2(2)}x^2 - c_{3(2)}x^3$. From $m_{G^2}(A(G^2)){\bf j}=0$, it follows that
\begin{eqnarray*}
A^4(G^2){\bf j} &=& c_{0(2)} {\bf j} + c_{1(2)}A(G^2){\bf j} + c_{2(2)}A^2(H^2){\bf j} + c_{3(2)}A^3(G^2){\bf j}\\
                &=& {\bf W}_{G^2} \left(\begin{array}{c}
                                       c_{0(2)}\\
                                       c_{1(2)}\\
                                       c_{2(2)}\\
                                       c_{3(2)}
                          \end{array}\right).
\end{eqnarray*}
Then, we may conclude that $c_{0(2)}=28$, $c_{1(2)}=48$, $c_{2(2)}=22$, $c_{3(2)}=0$ and thus $m_{G^2}(x) = x^4 - 28 - 48x - 22x^2$,
${\bf C}(m_{G^2}) = \left(\begin{array}{ccccc}
                                 0 &   0   &   0   & 28 \\
                                 1 &   0   &   0   & 48 \\
                                 0 &   1   &   0   & 22 \\
                                 0 &   0   &   1   &  0 \\
                          \end{array}\right)$ and ${\bf M}(G^2) = \left(\begin{array}{ccc}
                                                                   & {\bf j}^T W_{G^2} &   \\
                                                                 0 &    0 \qquad 0     & 0 \\
                                                                 0 &    0 \qquad 0     & 0 \\
                                                                 0 &    0 \qquad 0     & 0
                                                        \end{array}\right) = \left(\begin{array}{cccc}
                                                                              9 & 48 & 276 & 1528 \\
                                                                              0 & 0  &  0  &  0   \\
                                                                              0 & 0  &  0  &  0   \\
                                                                              0 & 0  &  0  &  0   \\
                                                                              \end{array}\right)$.                                        Therefore,
$$
\widetilde{\bf W}(G^2) = \left(\begin{array}{cccccccccccc}
                           0 &   0   &   0   & 28   & 9 & 48 & 276 & 1528 & 9 & 48 & 276 & 1528\\
                           1 &   0   &   0   & 48   & 0 & 0  &  0  &  0   & 0 &  0 &  0  &  0 \\
                           0 &   1   &   0   & 22   & 0 & 0  &  0  &  0   & 0 &  0 &  0  &  0 \\
                           0 &   0   &   1   &  0   & 0 & 0  &  0  &  0   & 0 &  0 &  0  &  0 \\
                           9 &  48   & 276   & 1528 & 0 & 0  &  0  & 28   & 0 &  0 &  0  &  0 \\
                           0 &   0   &   0   &  0   & 1 & 0  &  0  & 48   & 0 &  0 &  0  &  0 \\
                           0 &   0   &   0   &  0   & 0 & 1  &  0  & 22   & 0 &  0 &  0  &  0 \\
                           0 &   0   &   0   &  0   & 0 & 0  &  1  &  0   & 0 &  0 &  0  &  0 \\
                           9 &  48   & 276   & 1528 & 0 & 0  &  0  &  0   & 0 &  0 &  0  &  28 \\
                           0 &   0   &   0   &  0   & 0 & 0  &  0  &  0   & 1 &  0 &  0  &  48 \\
                           0 &   0   &   0   &  0   & 0 & 0  &  0  &  0   & 0 &  1 &  0  &  22 \\
                           0 &   0   &   0   &  0   & 0 & 0  &  0  &  0   & 0 &  0 &  1  &   0 \\
                           \end{array}\right).
$$
Applying Theorem~\ref{h-join_spectra_theorem} and taking into account that the main eigenvalues of $G^2$ are $\mu_{1(2)}, \mu_{2(2)}, \mu_{3(2)}, \mu_{4(2)}$ (the roots of $m_{G^2}(x)$) it follows that the spectrum of $G^{3}=G[G^2]$, is
\begin{equation*}
\sigma(G[G^2]) = \{\mu^{[0]}_{1(2)}, \mu^{[0]}_{2(2)}, \mu^{[0]}_{3(2)}, \mu^{[0]}_{4(2)}\} \cup \{-\sqrt{2}^{[3]}, 0^{[9]}, \sqrt{2}^{[3]}\}  \cup \sigma(\widetilde{\bf W}(H^2)).
\end{equation*}
It should be noted that
\begin{eqnarray*}
\phi(\widetilde{\bf W}(G^2)) \!\!\!&=&\!\!\! 67648 + 390144 x + 979904 x^2 + 1398912 x^3 + 1238704 x^4 + \\
                             & & 691392 x^5 + 230744 x^6 + 35712 x^7 - 2484 x^8 - 1872 x^9 - \\
                             & &228 x^{10} + x^{12}\\
                             \!\!\!&=&\!\!\! (-28 - 48 x - 22 x^2 + x^4)(-2416 - 9792 x - 16312 x^2 -\\
                             & & 14304 x^3 - 
                              6988 x^4 - 1824 x^5 - 206 x^6 + x^8).
\end{eqnarray*}
Therefore,
\begin{eqnarray*}
\phi(G^3) &=& x^9(-2 + x^2)^3(-28 - 48 x - 22 x^2 + x^4)(-2416 - 9792 x - \\
          & & 16312 x^2 - 14304 x^3 - 6988 x^4 - 1824 x^5 - 206 x^6 + x^8).
\end{eqnarray*}
Observe that the lexicographic power $G^3$ has order $27$.
\end{enumerate}

\end{example}

\section{Conclusions}
This paper provides a comprehensive analysis of the spectrum of the generalized composition of graphs, focusing particularly  on the lexicographic product $H[G]$.  Using the concept of $H$-join and introducing the $H[G]$ associated matrix $\widetilde{{\bf W}}$, which relates the graphs $H$ and $G$,  we determine the spectrum of $H[G]$ for an arbitrary graph $G$, in terms of the spectrum of $G$ and the spectrum of the matrix $\widetilde{{\bf W}}$. Additionally, some relationships between the main eigenvalues of $G$ and the eigenvalues of $\widetilde{{\bf W}}$, when $H$ has nullity $\eta > 0$, are established. Namely, when the nullity of $H$ is $\eta > 0$, it is proved that every main eigenvalue of $G$ is an eigenvalue of $\widetilde{{\bf W}}$ with multiplicity at least $\eta$ which is non-main for $\bf \widetilde{W}$ if and only if $0$ is non-main as an eigenvalue of $H$. Furthermore, the investigation into the spectra of lexicographic powers $G^k$  shows that these results can be applied to understand spectral evolution under repeated lexicographic products. These findings contribute to the broader field of spectral graph theory by offering valuable tools and insights for the further exploration of graph structures and their algebraic properties.
\bigskip

\noindent
{\bf Acknowledgments.} This research is supported by the Portuguese Foundation for Science	and Technology,	through the CIDMA - Center for Research and Development in Mathematics and Applications, within project UID/04106.

\end{document}